\documentclass[11pt]{article}

\usepackage{amsmath, amsthm}
\usepackage{amssymb}
\usepackage{eucal}
\usepackage{hyperref} 
\usepackage{mathrsfs}
\usepackage{scrextend}

\newtheorem{theorem}{Theorem}[section]
\newtheorem{lemma}[theorem]{Lemma}

\newtheorem{question}[theorem]{Question}
\newtheorem{theoremA}{Theorem}

\theoremstyle{definition}
\newtheorem{example}[theorem]{Example}
\newtheorem{notation}[theorem]{Notation}
\newtheorem{remark}[theorem]{Remark}

\begin{document}
	
	\baselineskip 16pt

	\title{Criteria for supersolvability of saturated fusion systems}
	
	\author{Fawaz Aseeri\\
		{\small Department of Mathematical Sciences, College of Applied Sciences, Umm Al-Qura University,}\\ {\small Makkah 21955, Saudi Arabia}\\
		{\small E-mail:
			fiaseeri@uqu.edu.sa}\\ \\
		{Julian Kaspczyk\footnote{Corresponding author} \footnote{At the date of submission, this author was not anymore affiliated with the Technische Universität Dresden. However, the bulk of the work presented here was done when he still was a postdoc at this university.}}\\
		{\small Institut für Algebra, Fakultät Mathematik, Technische Universität Dresden,}\\
		{\small 01069 Dresden, Germany}\\
		{\small E-mail: julian.kaspczyk@gmail.com}}
	
	\date{}
	\maketitle

	\begin{abstract} 
	Let $p$ be a prime number. A saturated fusion system $\mathcal{F}$ on a finite $p$-group $S$ is said to be \textit{supersolvable} if there is a series $1 = S_0 \le S_1 \le \dots \le S_m = S$ of subgroups of $S$ such that $S_i$ is strongly $\mathcal{F}$-closed for all $0 \le i \le m$ and such that $S_{i+1}/S_i$ is cyclic for all $0 \le i < m$. We prove some criteria that ensure that a saturated fusion system $\mathcal{F}$ on a finite $p$-group $S$ is supersolvable provided that certain subgroups of $S$ are abelian and weakly $\mathcal{F}$-closed. Our results can be regarded as generalizations of purely group-theoretic results of Asaad \cite{Asaad}. 
	\end{abstract}
	
	\footnotetext{Keywords: fusion systems, supersolvable, weakly closed, pronormal, weakly pronormal}

	\footnotetext{Mathematics Subject Classification (2020): 20D10, 20D20.}
	\let\thefootnote\thefootnoteorig

	\section{Introduction}

	In finite group theory, the term “fusion” refers to conjugacy relations between $p$-elements and $p$-subgroups. The study of fusion in finite groups has a long history, and many results concerning fusion in finite groups had a significant impact on finite group theory. Some well-known examples are Burnside's fusion theorem \cite[Lemma 5.12]{Isaacs}, Frobenius' $p$-nilpotency criterion \cite[Theorem 5.26]{Isaacs}, Alperin's fusion theorem \cite{Alperin} and Glauberman's $Z^{*}$-theorem \cite{Glauberman}.
	
	A modern approach to study problems concerning fusion in finite groups is the theory of fusion systems. The standard examples of fusion systems are the fusion categories of finite groups over $p$-subgroups. Given a prime number $p$, a finite group $G$ and a $p$-subgroup $S$ of $G$, the \textit{fusion category} of $G$ over $S$ is defined to be the category $\mathcal{F}_S(G)$ given as follows: The objects of $\mathcal{F}_S(G)$ are the subgroups of $S$, the morphisms in $\mathcal{F}_S(G)$ are the group homomorphisms between subgroups of $S$ induced by conjugation in $G$, and the composition of morphisms in $\mathcal{F}_S(G)$ is the usual composition of group homomorphisms. Abstract fusion systems can be regarded as a generalization of this concept. Given a prime number $p$ and a finite $p$-group $S$, a \textit{fusion system} on $S$ is a category $\mathcal{F}$ whose objects are the subgroups of $S$ and whose morphisms behave as if they are induced by conjugation inside a finite group containing $S$ as a $p$-subgroup (see \cite[Part I, Definition 2.1]{AKO}). A fusion system is said to be \textit{saturated} if it satisfies certain additional axioms (see \cite[Part I, Definition 2.2]{AKO}). Any fusion category of a finite group over a Sylow subgroup is saturated (see \cite[Part I, Theorem 2.3]{AKO}), but not every saturated fusion system appears as the fusion category of a finite group over a Sylow subgroup (see \cite[Part III, Section 6]{AKO}). If $S$ is a Sylow $p$-subgroup of a finite group $G$ for some prime $p$, then we refer to $\mathcal{F}_S(G)$ as the \textit{$p$-fusion system} of $G$. 
	
	The book \cite{AKO} provides a detailed introduction to the theory of fusion systems, and the reader is asked to consult that book for any definitions on fusion systems we do not explain here. For unfamiliar definitions on groups, the reader is referred to \cite{Huppert, Isaacs}. 
	
	
	
	This paper is concerned with the concept of a supersolvable saturated fusion system introduced in \cite{Su}. Given a prime $p$ and a finite $p$-group $S$, a saturated fusion system $\mathcal{F}$ on $S$ is said to be \textit{supersolvable} if there is a series $1 = S_0 \le S_1 \le \dots \le S_m = S$ of subgroups of $S$ such that $S_i$ is strongly $\mathcal{F}$-closed for all $0 \le i \le m$ and such that $S_{i+1}/S_i$ is cyclic for all $0 \le i < m$ (see \cite[Definition 1.2]{Su}). By \cite[Proposition 1.3]{Su}, for any prime $p$, the supersolvable saturated fusion systems on finite $p$-groups are precisely the $p$-fusion systems of supersolvable finite groups, and they are also precisely the $p$-fusion systems of $p$-supersolvable finite groups. After their introduction in \cite{Su}, supersolvable saturated fusion systems (and related concepts) were further studied in the papers \cite{BB, ShenZhang}.  
	
	The goal of this paper is to obtain criteria for supersolvability of saturated fusion systems. Our results are inspired from a current line of research in finite group theory. Namely, there is currently very active research on subgroup embedding properties, and a problem of particular interest is to study the structure of a finite group $G$ under the assumption that some given subgroups of $G$ satisfy a given embedding property. In this context, many results of the following form were obtained: Given a fixed prime $p$ and a finite group $G$, it is assumed that all (or at least sufficiently many) $p$-subgroups of $G$ with some fixed order satisfy a certain embedding property. Sometimes, a number of additional conditions are assumed, and the conclusion usually is that $G$ is $p$-supersolvable or $p$-nilpotent. As an example, we mention a result of Berkovich and Isaacs. Given a prime $p$, a natural number $e \ge 3$ and a finite group $G$ with a noncyclic Sylow $p$-subgroup of order exceeding $p^e$, they proved that $G$ is $p$-supersolvable if any noncyclic subgroup of $G$ with order $p^e$ is normal in $G$ (see \cite[Theorem C]{BerkovichIsaacs}). Other results of this kind were obtained, for example, in the papers \cite{Asaad_c_supplemented, Asaad, GuoIsaacs, Kaspczyk, LiuYu, Yu}.
	
	
	Our principal motivation is to prove supersolvability criteria of a similar spirit for saturated fusion systems. Indeed, given a prime $p$ and a saturated fusion system $\mathcal{F}$ on a finite $p$-group $S$, we will prove criteria ensuring that $\mathcal{F}$ is supersolvable provided that all subgroups of $S$ with some fixed order are “suitably embedded” in $S$ with respect to $\mathcal{F}$. More precisely, we will assume that all subgroups of $S$ with some fixed order are weakly $\mathcal{F}$-closed and that sufficiently many of them are abelian.
	
	Our results generalize some purely group-theoretic results of Asaad \cite{Asaad}, and before stating our results, we consider the corresponding results from \cite{Asaad}. 
	
	First, let us recall some definitions. A subgroup $H$ of a group $G$ is said to be \textit{pronormal} in $G$ if $H$ and $H^g$ are conjugate in $\langle H, H^g \rangle$ for each $g \in G$. By \cite[Theorem 4.3]{LiuYu}, if $S$ is a Sylow $p$-subgroup of a finite group $G$ for some prime $p$, then a subgroup $Q$ of $S$ is pronormal in $G$ if and only if $Q$ is weakly $\mathcal{F}_S(G)$-closed. Following Asaad \cite{Asaad}, we say that a subgroup $H$ of a group $G$ is \textit{weakly pronormal} in $G$ if there is a subgroup $K$ of $G$ such that $G = HK$ and such that $H \cap K$ is pronormal in $G$. Note that, if $S$ is a Sylow $p$-subgroup of a finite group $G$ for some prime $p$, a subgroup $Q$ of $S$ is weakly pronormal in $G$ if and only if there is a subgroup $K$ of $G$ such that $G = QK$ and such that $Q \cap K$ is weakly $\mathcal{F}_S(G)$-closed.
	
	The following theorem is one of the main results of \cite{Asaad}. 
	
	\begin{theorem}
		\label{theorem_asaad_one} 
	(\cite[Theorem 1.3]{Asaad}) Let $G$ be a nontrivial finite group, $p$ be the smallest prime dividing $|G|$ and $S$ be a noncyclic Sylow $p$-subgroup of $G$. Suppose that there is a subgroup $D$ of $S$ with $1 < D < S$ such that any subgroup of $S$ with order $|D|$ or $p|D|$ is abelian and weakly pronormal in $G$. Then $G$ is $p$-nilpotent. 
	\end{theorem}

Note that, because of Burnside's $p$-nilpotency criterion \cite[Corollary 5.14]{Isaacs}, the condition in Theorem \ref{theorem_asaad_two} that $S$ is noncyclic is not necessary. 

We will prove the following generalization of Theorem \ref{theorem_asaad_one}.

\begin{theoremA}
	\label{first_main_result} 
	Let $G$ be a finite group, $p$ be a prime and $S$ be a Sylow $p$-subgroup of $G$. Suppose that there is a subgroup $D$ of $S$ with $1 < D < S$ such that any subgroup of $S$ with order $|D|$ or $p|D|$ is abelian and weakly pronormal in $G$. Then $\mathcal{F}_S(G)$ is supersolvable.  
\end{theoremA}

By \cite[Theorem 1.9 (a)]{Su}, if $G$ is a finite group and if $S$ is a Sylow $p$-subgroup of $G$ for some prime $p$ with $(|G|,p-1) = 1$, then $G$ is $p$-nilpotent provided that $\mathcal{F}_S(G)$ is supersolvable. Of course, the condition $(|G|,p-1) = 1$ is satisfied when $p$ is the smallest prime divisor of $|G|$, and hence, Theorem \ref{theorem_asaad_one} is covered by Theorem \ref{first_main_result}.

When the hypotheses of Theorem \ref{first_main_result} are satisfied, it is not necessarily true that $G$ is $p$-supersolvable. For example, let $G := A_5 \times A_5$ and $S$ be a Sylow $5$-subgroup of $G$. Then $S \cong C_5 \times C_5$ is abelian. Also, any subgroup of $S$ with order $5$ and $S$ itself are complemented in $G$ and hence weakly pronormal in $G$. However, $G$ is not $5$-supersolvable.



Our further results are supersolvability criteria for abstract saturated fusion systems, and they are motivated by the following result of Asaad \cite{Asaad}. 

\begin{theorem}
\label{theorem_asaad_two} (\cite[Corollary 3.1]{Asaad}) Let $G$ be a nontrivial finite group, $p$ be the smallest prime dividing $|G|$ and $S$ be a noncyclic Sylow $p$-subgroup of $G$. Suppose that there is a subgroup $D$ of $S$ with $1 < D < S$ such that any subgroup of $S$ with order $|D|$ is abelian and pronormal in $G$. If $S$ is a nonabelian $2$-group, suppose moreover that any subgroup of $S$ with order $2|D|$ is abelian and pronormal in $G$. Then $G$ is $p$-nilpotent. 
\end{theorem}

Note that, because of Burnside's $p$-nilpotency criterion \cite[Corollary 5.14]{Isaacs}, the condition in Theorem \ref{theorem_asaad_two} that $S$ is noncyclic is not necessary. 

Let $G$ be a nontrivial finite group, $p$ be the smallest prime divisor of $|G|$ and $S$ be a Sylow $p$-subgroup of $G$. As remarked before Theorem \ref{theorem_asaad_one}, a subgroup $Q$ of $S$ is pronormal in $G$ if and only if $Q$ is weakly $\mathcal{F}_S(G)$-closed. Also, $G$ is $p$-nilpotent if and only if $\mathcal{F}_S(G)$ is supersolvable (indeed, if $\mathcal{F}_S(G)$ is supersolvable, then $G$ is $p$-nilpotent by the remark following Theorem \ref{first_main_result}, and conversely, if $G$ is $p$-nilpotent, then $\mathcal{F}_S(G) = \mathcal{F}_S(S)$ is supersolvable). Consequently, we can reformulate Theorem \ref{theorem_asaad_two} as follows: 

\medskip
\begin{addmargin}[25pt]{0pt}
	\textit{Let $G$ be a nontrivial finite group, $p$ be the smallest prime dividing $|G|$ and $S$ be a noncyclic Sylow $p$-subgroup of $G$. Suppose that there is a subgroup $D$ of $S$ with $1 < D < S$ such that any subgroup of $S$ with order $|D|$ is abelian and weakly $\mathcal{F}_S(G)$-closed. If $S$ is a nonabelian $2$-group, suppose moreover that any subgroup of $S$ with order $2|D|$ is abelian and weakly $\mathcal{F}_S(G)$-closed. Then $\mathcal{F}_S(G)$ is supersolvable.} 
\end{addmargin} 

\medskip
In view of this reformulation of Theorem \ref{theorem_asaad_two}, the following question naturally arises. 

\begin{question}
\label{question_one} 
Let $p$ be a prime number and $\mathcal{F}$ be a saturated fusion system on a finite $p$-group $S$. Suppose that there is a subgroup $D$ of $S$ with $1 < D < S$ such that any subgroup of $S$ with order $|D|$ is abelian and weakly $\mathcal{F}$-closed. If $S$ is a nonabelian $2$-group, suppose moreover that any subgroup of $S$ with order $2|D|$ is abelian and weakly $\mathcal{F}$-closed. Is it true in general that $\mathcal{F}$ is supersolvable? 
\end{question}

Our next result positively answers Question \ref{question_one} for the case $p = 2$. 

\begin{theoremA}
\label{second_main_result} 
Let $\mathcal{F}$ be a saturated fusion system on a finite $2$-group $S$. Suppose that there is a subgroup $D$ of $S$ with $1 < D < S$ such that any subgroup of $S$ with order $|D|$ is abelian and weakly $\mathcal{F}$-closed. If $S$ is nonabelian, suppose moreover that any subgroup of $S$ with order $2|D|$ is abelian and weakly $\mathcal{F}$-closed. Then $\mathcal{F}$ is supersolvable. 
\end{theoremA}

From \cite[Proposition 1.3]{Su}, we see that a saturated fusion system $\mathcal{F}$ on a finite $2$-group $S$ is supersolvable if and only if $\mathcal{F}$ is nilpotent, i.e. if and only if $\mathcal{F} = \mathcal{F}_S(S)$. Therefore, we could replace the word “supersolvable” by the word “nilpotent” in Theorem \ref{second_main_result}.  

The following result gives a positive answer to Question \ref{question_one} for the case that $p$ is odd and that $D$ is maximal in $S$. In fact, it shows that even more is true. 

\begin{theoremA}
\label{third_main_result}
Let $p$ be an odd prime number and $\mathcal{F}$ be a saturated fusion system on a finite $p$-group $S$. Suppose that any maximal subgroup of $S$ is weakly $\mathcal{F}$-closed. If $S$ is not cyclic, suppose moreover that $S$ has more than one abelian maximal subgroup. Then $\mathcal{F}$ is supersolvable. 
\end{theoremA}

In the statement of Theorem \ref{third_main_result}, it would not be enough to assume that $S$ has \textit{at least} one abelian maximal subgroup when $S$ is not cyclic. This can be seen from the following example. 

\begin{example}
Let $G$ be the group indexed in GAP \cite{GAP} as \verb|SmallGroup(324,160)|, $S$ be a Sylow $3$-subgroup of $G$ and $\mathcal{F} := \mathcal{F}_S(G)$. Then $S$ has precisely four maximal subgroups, precisely one of them is abelian, and all of them are weakly $\mathcal{F}$-closed. However, $\mathcal{F}$ is not supersolvable. Indeed, since $G$ is solvable, the supersolvability of $\mathcal{F}$ would imply that $G$ is $3$-supersolvable (see \cite[Theorem 1.9 (b)]{Su}), but one can check by using GAP \cite{GAP} that $G$ is not $3$-supersolvable.
\end{example} 

The next result gives a complete positive answer to Question \ref{question_one} for the case that $p$ is odd. 

\begin{theoremA}
\label{fourth_main_result} 
Let $p$ be an odd prime number and $\mathcal{F}$ be a saturated fusion system on a finite $p$-group $S$. Suppose that there is a subgroup $D$ of $S$ with $1 < D < S$ such that any subgroup of $S$ with order $|D|$ is abelian and weakly $\mathcal{F}$-closed. Then $\mathcal{F}$ is supersolvable. 
\end{theoremA}

We remark that our proof of Theorem \ref{fourth_main_result} indirectly relies on the classification of finite simple groups. More precisely, in our proof of Theorem \ref{fourth_main_result}, we will apply \cite[Theorem 3.5]{Kappe}, and the classification of finite simple groups was used in the proof of that result.

We finish the introduction with the following open question that naturally arises in view of Theorems \ref{third_main_result} and \ref{fourth_main_result}.

\begin{question}
\label{open_question_one} 
Let $p$ be an odd prime number and $\mathcal{F}$ be a saturated fusion system on a finite $p$-group $S$. Suppose that there is a subgroup $D$ of $S$ with $1 < D < S$ such that any subgroup of $S$ with order $|D|$ is weakly $\mathcal{F}$-closed. If $S$ is not cyclic, suppose moreover that $S$ has more than one abelian subgroup with order $|D|$. Is it true in general that $\mathcal{F}$ is supersolvable? 
\end{question}





 
\section{Preliminaries} 
\label{preliminaries} 
In this section, we collect some lemmas needed for the proofs of our main results. 

\begin{lemma}
\label{lemma_strongly_closed_1} 
Let $p$ be a prime number, $\mathcal{F}$ be a fusion system on a finite $p$-group $S$ and $Q \le R \le S$. Suppose that $Q$ is strongly $\mathcal{F}$-closed and that $R$ is weakly $\mathcal{F}$-closed. Then $R/Q$ is weakly $\mathcal{F}/Q$-closed. 
\end{lemma}

\begin{proof}
Let $\psi: R/Q \rightarrow S/Q$ be a morphism in $\mathcal{F}/Q$. We have to show that $\psi(R/Q) = R/Q$. As an $\mathcal{F}/Q$-morphism, $\psi$ is induced by an $\mathcal{F}$-morphism, i.e. there is a morphism $\varphi \in \mathrm{Hom}_{\mathcal{F}}(R,S)$ such that $\psi(rQ) = \varphi(r)Q$ for all $r \in R$. Since $R$ is weakly $\mathcal{F}$-closed, we have $\psi(R/Q) = \varphi(R)Q/Q = RQ/Q = R/Q$, as required. 
\end{proof}

\begin{lemma}
\label{lemma_strongly_closed_2}
	Let $p$ be a prime number, $\mathcal{F}$ be a fusion system on a finite $p$-group $S$ and $Q \le R \le S$. Suppose that $Q \trianglelefteq \mathcal{F}$ and that $R/Q$ is strongly $\mathcal{F}/Q$-closed. Then $R$ is strongly $\mathcal{F}$-closed. 
\end{lemma}

\begin{proof}
	Let $P$ be a subgroup of $R$ and $\varphi: P \rightarrow S$ be a morphism in $\mathcal{F}$. We have to show that $\varphi(P) \le R$. Since $Q$ is normal in $\mathcal{F}$ by hypothesis, $\varphi$ extends to a morphism $\psi \in \mathrm{Hom}_{\mathcal{F}}(PQ,S)$ with $\psi(Q) = Q$. Let 
	\begin{equation*}
		\overline{\psi}: PQ/Q \rightarrow S/Q, xQ \mapsto \psi(x)Q. 
	\end{equation*}
	Then $\overline{\psi}$ is a morphism in $\mathcal{F}/Q$. Since $R/Q$ is strongly $\mathcal{F}/Q$-closed and $PQ/Q \le R/Q$, it follows that $\overline{\psi}(PQ/Q) \le R/Q$. So, we have $\varphi(P) = \psi(P) \le R$, as required. 
\end{proof}

\begin{lemma}
\label{lemma_strongly_closed_3}
Let $p$ be a prime number, $\mathcal{F}$ be a supersolvable saturated fusion system on a finite $p$-group $S$ and $Q$ be a weakly $\mathcal{F}$-closed subgroup of $S$. Then $Q$ is strongly $\mathcal{F}$-closed. 
\end{lemma}

\begin{proof}
Let $P$ be a subgroup of $Q$ and $\varphi: P \rightarrow S$ be a morphism in $\mathcal{F}$. We have to show that $\varphi(P) \le Q$. By \cite[Proposition 2.3]{Su}, $S$ is normal in $\mathcal{F}$. Hence, $\varphi$ extends to an automorphism $\psi \in \mathrm{Aut}_{\mathcal{F}}(S)$. Since $Q$ is weakly $\mathcal{F}$-closed, we have $\varphi(P) = \psi(P) \le \psi(Q) = Q$, as required. 
\end{proof}

\begin{lemma}
\label{lemma_supersolvable_FS_2}
Let $p$ be a prime number, $\mathcal{F}$ be a supersolvable saturated fusion system on a finite $p$-group $S$ and $Q$ be a strongly $\mathcal{F}$-closed subgroup of $S$. Then the following hold:
\begin{enumerate}
\item[\textnormal{(1)}] There is a series 
\begin{equation*}
1 = S_0 \le S_1 \le \dots \le S_m = S
\end{equation*}
of subgroups of $S$ such that $S_i$ is strongly $\mathcal{F}$-closed for all $0 \le i \le m$, such that $S_{i+1}/S_i$ is cyclic for all $0 \le i < m$ and such that $S_j = Q$ for some $0 \le j \le m$. 
\item[\textnormal{(2)}] $\mathcal{F}/Q$ is supersolvable. 
\end{enumerate} 
\end{lemma}

\begin{proof}
By \cite[Proposition 1.3]{Su}, there is a $p$-supersolvable finite group $G$ with $S \in \mathrm{Syl}_p(G)$ and $\mathcal{F} = \mathcal{F}_S(G)$. Since $S$ is normal in $\mathcal{F}$ by \cite[Proposition 2.3]{Su}, we have $\mathcal{F} = N_{\mathcal{F}}(S) = \mathcal{F}_S(N_G(S))$. Therefore, upon replacing $G$ by $N_G(S)$, we may (and will) assume that $S \trianglelefteq G$. 

Since $Q$ is strongly $\mathcal{F}$-closed and $S \trianglelefteq G$, we have $Q \trianglelefteq G$. Let $1 = H_0 \le H_1 \le \dots \le H_n = G$ be a chief series of $G$ through $Q$ and $S$. Then $S = H_m$ for some $0 \le m \le n$ and $Q = H_j$ for some $0 \le j \le m$. 

Set $S_i := H_i$ for all $0 \le i \le m$. Then we have $1 = S_0 \le S_1 \le \dots \le S_m = S$. For each $0 \le i \le m$, the subgroup $S_i$ is strongly $\mathcal{F}$-closed since $S_i \trianglelefteq G$. For each $0 \le i < m$, the group $S_{i+1}/S_i$ is cyclic of order $p$ since it is a chief factor of the $p$-supersolvable group $G$. Moreover, $S_j = H_j = Q$. The proof of (1) is now complete. 

It remains to prove (2). We have 
\begin{equation*}
	1 = S_j/Q \le S_{j+1}/Q \le \dots \le S_m/Q = S/Q. 
\end{equation*}
For each $j \le k \le m$, we have $S_k/Q \trianglelefteq G/Q$, and since $\mathcal{F}/Q = \mathcal{F}_{S/Q}(G/Q)$, it follows that $S_k/Q$ is strongly $\mathcal{F}/Q$-closed. Moreover, for each $j \le k < m$, the quotient $(S_{k+1}/Q)/(S_k/Q) \cong S_{k+1}/S_k$ is cyclic of order $p$. It follows that $\mathcal{F}/Q$ is supersolvable, and so the proof of (2) is complete. 
\end{proof} 

To state the next lemma, which will play a key role in the proofs of all our main results, we introduce the following notation.

\begin{notation}
	Let $p$ be a prime number and $\mathcal{F}$ be a saturated fusion system on a finite $p$-group $S$. We set
	\begin{equation*}
		\mathcal{E}_{\mathcal{F}}^{*} := \lbrace Q \le S \mid \textnormal{$Q$ is $\mathcal{F}$-essential, or $Q = S$} \rbrace.
	\end{equation*}
\end{notation} 

\begin{lemma}
	\label{criterion_FS_one}
	Let $p$ be a prime number and $\mathcal{F}$ be a saturated fusion system on a finite $p$-group $S$. Suppose that, for each $Q \in \mathcal{E}_{\mathcal{F}}^{*}$, the fusion system $N_{\mathcal{F}}(Q)$ is supersolvable. Then $\mathcal{F}$ is supersolvable. 
\end{lemma}

\begin{proof}
	Let $Q \in \mathcal{E}_{\mathcal{F}}^{*}$. Since $N_{\mathcal{F}}(Q)$ is supersolvable, $\mathrm{Aut}_{N_{\mathcal{F}}(Q)}(Q) = \mathrm{Aut}_{\mathcal{F}}(Q)$ is $p$-closed by \cite[Proposition 1.3]{Su}. Hence, $\mathrm{Out}_{\mathcal{F}}(Q)$ is $p$-closed. From \cite[Proposition A.7 (c)]{AKO}, we see that a $p$-closed finite group cannot possess a strongly $p$-embedded subgroup. Consequently, $\mathrm{Out}_{\mathcal{F}}(Q)$ does not possess a strongly $p$-embedded subgroup, and so $Q$ is not $\mathcal{F}$-essential. It follows that $\mathcal{E}_{\mathcal{F}}^{*} = \lbrace S \rbrace$. Applying \cite[Part I, Proposition 4.5]{AKO}, we conclude that $S \trianglelefteq \mathcal{F}$. The proof is now complete since $N_{\mathcal{F}}(S) = \mathcal{F}$ is supersolvable by hypothesis.
\end{proof}

\begin{lemma}
	\label{criterion_FS_zero}
	Let $p$ be a prime number, $\mathcal{F}$ be a saturated fusion system on a finite $p$-group $S$ and $Q \trianglelefteq \mathcal{F}$. Suppose that there is a series $1 = Q_0 \le Q_1 \le \dots \le Q_n = Q$ of subgroups of $Q$ such that $Q_i$ is strongly $\mathcal{F}$-closed for all $0 \le i \le n$ and such that $Q_{i+1}/Q_i$ is cyclic for all $0 \le i < n$. Suppose moreover that $\mathcal{F}/Q$ is supersolvable. Then $\mathcal{F}$ is supersolvable.
\end{lemma}

\begin{proof}
	Since $\mathcal{F}/Q$ is supersolvable, there is a series $Q = V_0 \le V_1 \le \dots \le V_m = S$ of subgroups of $S$ such that $V_i/Q$ is strongly $\mathcal{F}/Q$-closed for all $0 \le i \le m$ and such that $(V_{i+1}/Q)/(V_i/Q) \cong V_{i+1}/V_i$ is cyclic for all $0 \le i < m$. By Lemma \ref{lemma_strongly_closed_2}, $V_i$ is strongly $\mathcal{F}$-closed for all $0 \le i \le m$. Set $S_i := Q_i$ for all $0 \le i \le n$ and $S_{n+\ell} := V_{\ell}$ for all $1 \le \ell \le m$. Then $1 = S_0 \le S_1 \le \dots \le S_{n+m} = S$. Also, $S_i$ is strongly $\mathcal{F}$-closed for all $0 \le i \le n+m$, and $S_{i+1}/S_i$ is cyclic for all $0 \le i < n+m$. Consequently, $\mathcal{F}$ is supersolvable. 
\end{proof}

For the next lemma, we recall that the \textit{supersolvable hypercentre} of a finite group $G$, denoted by $Z_{\mathcal{U}}(G)$, is the largest normal subgroup of $G$ such that every $G$-chief factor below it is cyclic. 

\begin{lemma}
\label{lemma_skiba} 
Let $p$ be a prime number and $P$ be a nontrivial normal $p$-subgroup of a finite group $G$. Suppose that $P$ has a subgroup $D$ with $1 < D < P$ such that every subgroup of $P$ of order $|D|$ is normal in $G$. If $P$ is a nonabelian $2$-group and $|D| = 2$, suppose moreover that every cyclic subgroup of $P$ of order $4$ is normal in $G$. Then $P \le Z_{\mathcal{U}}(G)$.
\end{lemma}

\begin{proof}
This follows from \cite{Skiba}. 
\end{proof}

\begin{lemma}
	\label{criterion_FS_two}
Let $G$ be a finite group, $p$ be a prime and $S$ be a Sylow $p$-subgroup of $G$. Suppose that, for any proper subgroup $H$ of $G$ with $O_p(G) < S \cap H$ and $S \cap H \in \mathrm{Syl}_p(H)$, the fusion system $\mathcal{F}_{S \cap H}(H)$ is supersolvable. Suppose moreover that $O_p(G) \le Z_{\mathcal{U}}(G)$. Then $\mathcal{F}_S(G)$ is supersolvable. 
\end{lemma}

\begin{proof}
Set $\mathcal{F} := \mathcal{F}_S(G)$, $\overline G := G/O_p(G)$ and $\overline{\mathcal{F}} := \mathcal{F}_{\overline S}(\overline G)$. 

Since $O_p(G) \le Z_{\mathcal{U}}(G)$, there is a series $1 = U_0 \le U_1 \le \dots \le U_n = O_p(G)$ of subgroups of $O_p(G)$ such that $U_i \trianglelefteq G$ for all $0 \le i \le n$ and such that $U_{i+1}/U_i$ is cyclic for all $0 \le i < n$. In particular, $U_i$ is strongly $\mathcal{F}$-closed for all $0 \le i \le n$. So, by Lemma \ref{criterion_FS_zero}, it suffices to show that $\mathcal{F}/O_p(G) = \overline{\mathcal{F}}$ is supersolvable. This is clear when $S = O_p(G)$, and so we assume that $O_p(G) \ne S$. 

Let $O_p(G) < Q \le S$ such that $\overline Q$ is fully $\overline{\mathcal{F}}$-normalized. Then $N_{\overline S}(\overline Q) = \overline{N_S(Q)}$ is a Sylow $p$-subgroup of $N_{\overline G}(\overline Q) = \overline{N_G(Q)}$. Hence, $N_S(Q)$ is a Sylow $p$-subgroup of $N_G(Q)$. Since $O_p(G) < Q$, we have $N_G(Q) < G$. Moreover, $O_p(G) < Q \le S \cap N_G(Q)$, and $S \cap N_G(Q) = N_S(Q)$ is a Sylow $p$-subgroup of $N_G(Q)$. So, by hypothesis, the fusion system $N_{\mathcal{F}}(Q) = \mathcal{F}_{N_S(Q)}(N_G(Q))$ is supersolvable. Lemma \ref{lemma_supersolvable_FS_2} (2) implies that $N_{\mathcal{F}}(Q)/O_p(G) = \mathcal{F}_{\overline{N_S(Q)}}(\overline{N_G(Q)}) = N_{\overline{\mathcal{F}}}(\overline Q)$ is supersolvable.

Now, let $O_p(G) \le R \le S$ such that $\overline{R} \in \mathcal{E}_{\overline{\mathcal F}}^{*}$. Then $\overline R$ is $\overline{\mathcal F}$-centric and fully $\overline{\mathcal F}$-normalized. Since $\overline S \ne 1$, we have $\overline R \ne 1$ and thus $O_p(G) < R$, and the preceding paragraph implies that $N_{\overline{\mathcal F}}(\overline R)$ is supersolvable. Lemma \ref{criterion_FS_one} yields that $\overline{\mathcal F}$ is supersolvable, as required. 
\end{proof}



\begin{lemma}
\label{criterion_FS_three}
Let $p$ be a prime number and $\mathcal{F}$ be a saturated fusion system on a finite $p$-group $S$. Suppose that, whenever $R$ is a subgroup of $S$ with $O_p(\mathcal{F}) < R$, any proper saturated subsystem of $\mathcal{F}$ on $R$ is supersolvable. Suppose moreover that there is a series $1 = S_0 \le S_1 \le \dots \le S_n = O_p(\mathcal{F})$ of subgroups of $O_p(\mathcal{F})$ such that $S_i$ is strongly $\mathcal{F}$-closed for all $0 \le i \le n$ and such that $S_{i+1}/S_i$ is cyclic for all $0 \le i < n$. Then $\mathcal{F}$ is supersolvable. 
\end{lemma}

\begin{proof}
Set $\overline S := S/O_p(\mathcal{F})$ and $\overline{\mathcal{F}} := \mathcal{F}/O_p(\mathcal{F})$. By Lemma \ref{criterion_FS_zero}, it suffices to show that $\overline{\mathcal{F}}$ is supersolvable. This is clear when $O_p(\mathcal{F}) = S$, and so we assume that $O_p(\mathcal{F}) \ne S$. 	
	
Let $O_p(\mathcal{F}) < Q \le S$ such that $\overline{Q}$ is fully $\overline{\mathcal F}$-normalized. Then, by \cite[Proposition 5.58 (iii)]{Craven}, $Q$ is fully $\mathcal{F}$-normalized. Hence, $N_{\mathcal{F}}(Q)$ is a saturated subsystem of $\mathcal{F}$ on $N_S(Q)$. Since $O_p(\mathcal{F}) < Q$, we have $O_p(\mathcal{F}) < N_S(Q)$ and $N_{\mathcal{F}}(Q) \ne \mathcal{F}$. So, by hypothesis, the fusion system $N_{\mathcal{F}}(Q)$ is supersolvable. Lemma \ref{lemma_supersolvable_FS_2} (2) implies that $N_{\mathcal{F}}(Q)/O_p(\mathcal{F})$ is supersolvable. By \cite[Exercise 5.11]{Craven}, we have $N_{\mathcal{F}}(Q)/O_p(\mathcal{F}) = N_{\overline{\mathcal{F}}}(\overline Q)$, and so $N_{\overline{\mathcal{F}}}(\overline Q)$ is supersolvable. 

Let $O_p(\mathcal{F}) \le R \le S$ such that $\overline{R} \in \mathcal{E}_{\overline{\mathcal F}}^{*}$. Then $\overline R$ is fully $\overline{\mathcal{F}}$-normalized and $\overline{\mathcal{F}}$-centric. Since $\overline S \ne 1$, we have $\overline R \ne 1$ and thus $O_p(\mathcal F) < R$, and the preceding paragraph implies that $N_{\overline{\mathcal F}}(\overline R)$ is supersolvable. Lemma \ref{criterion_FS_one} yields that $\overline{\mathcal{F}}$ is supersolvable, as required.
\end{proof}

\begin{lemma}
\label{lemma_extraspecial} 
Let $p$ be an odd prime number and $SL_2(p) \le H \le GL_2(p)$. Let $V$ denote the group consisting of all row vectors $\begin{pmatrix} x & y \end{pmatrix}$, where $x,y \in \mathbb{F}_p$, together with the componentwise addition. Moreover, let $G = V \rtimes H$ be the outer semidirect product of $V$ and $H$ with respect to the natural action of $H$ on $V$, i.e. $G$ is the group consisting of all pairs $(h,v)$ with $h \in H$ and $v \in V$, together with the multiplication given by 
\begin{equation*} 
	(h_1,v_1) \cdot (h_2,v_2) := (h_1h_2, v_1h_2+v_2)
\end{equation*} 
for all $h_1,h_2 \in H$, $v_1,v_2 \in V$. Let 
\begin{equation*}
U := \left\lbrace \begin{pmatrix} 1 & a \\ 0 & 1 \end{pmatrix} \ : \ a \in \mathbb{F}_p \right\rbrace \le SL_2(p) \le H
\end{equation*}
and 
\begin{equation*} 
S := \lbrace (u,v) \ : \ u \in U, v \in V \rbrace.	
\end{equation*} 
Then $S$ is a Sylow $p$-subgroup of $G$, and there is a maximal subgroup of $S$ which is not weakly $\mathcal{F}_S(G)$-closed. 
\end{lemma}

\begin{proof}
Since  $|SL_2(p)| = (p^2-1)p$ by \cite[Kapitel II, Hilfssatz 6.2 (2)]{Huppert} and since $|H : SL_2(p)|$ is not divisible by $p$, the largest power of $p$ dividing $|H|$ is $p$. Therefore, $p^3$ is the largest power of $p$ dividing $|G|$. Clearly, $S$ is a subgroup of $G$ of order $p^3$. Thus $S \in \mathrm{Syl}_p(G)$. 

Let 
\begin{equation*}
A := \begin{pmatrix} 1 & 1 \\ 0 & 1 \end{pmatrix} \in U, \ \ \ \ B := I_2 = \begin{pmatrix} 1 & 0 \\ 0 & 1 \end{pmatrix} \in U,
\end{equation*}
and
\begin{equation*}
a := \begin{pmatrix} 1 & 0 \end{pmatrix} \in V, \ \ \ \, b := \begin{pmatrix} 0 & 1 \end{pmatrix} \in V.  
\end{equation*}
Set $r := (A,a) \in S$ and $s := (B,b) \in S$. A direct calculation shows that $|\langle r \rangle| = |\langle s \rangle| = p$, $rs = sr$ and $\langle r \rangle \cap \langle s \rangle = 1$. Therefore, $S_1 := \langle r,s \rangle = \langle r \rangle \langle s \rangle$ has order $p^2$, whence $S_1$ is maximal in $S$. 

We show now that $S_1$ is not weakly $\mathcal{F}_S(G)$-closed. Let 
\begin{equation*} 
M := \begin{pmatrix} -1 & 1 \\ 0 & -1 \end{pmatrix} \in SL_2(p) \le H
\end{equation*} 
and $g := (M,0_V) \in G$. Then
\begin{equation*} 
g^{-1}rg = (A,b-a) \in S
\end{equation*}
and 
\begin{equation*} 
g^{-1}sg = (B,-b) \in S.
\end{equation*} 
Thus $g^{-1} S_1 g = \langle g^{-1} r g, g^{-1} s g \rangle \le S$. Assume that $S_1 = g^{-1}S_1g$. Then $g^{-1}rg \in S_1$, and so there exist $0 \le i, j < p$ with $g^{-1}rg = r^i s^j$. Then $A = A^i$ and hence $i = 1$. It follows that $(A,b-a) = g^{-1}rg = rs^j = (A,a+jb)$, whence $\begin{pmatrix} -1 & 1 \end{pmatrix} = b-a = a+jb = \begin{pmatrix} 1 & j \end{pmatrix}$ and thus $-1 = 1$. This is a contradiction since $p$ is odd. So we have $S_1 \ne g^{-1}S_1g \le S$, and this shows that $S_1$ is not weakly $\mathcal{F}_S(G)$-closed. 
\end{proof}

The next lemma is certainly well-known, but we include a proof for the convenience of the reader. 

\begin{lemma}
	\label{Sylow_GL} 
	Let $p$ be an odd prime number and $n \in \lbrace 2,3 \rbrace$. Then the following hold: 
	\begin{enumerate}
		\item[\textnormal{(1)}] The Sylow $p$-subgroups of $GL_n(p)$ have exponent $p$.
		\item[\textnormal{(2)}] If $H$ is a nontrivial cyclic $p$-subgroup of $GL_n(p)$, then $H \cong C_p$.
	\end{enumerate}
\end{lemma}

\begin{proof}
By \cite[Kapitel II, Hilfssatz 6.2 (1)]{Huppert}, we have $|GL_2(p)| = (p^2-1)(p^2-p) = p(p^2-1)(p-1)$. Hence, the Sylow $p$-subgroups of $GL_2(p)$ have order $p$. In particular, they have exponent $p$, and so (1) holds for $n = 2$. 

Again by \cite[Kapitel II, Hilfssatz 6.2 (1)]{Huppert}, we have $|GL_3(p)| = (p^3-1)(p^3-p)(p^3-p^2) = p^3(p^3-1)(p^2-1)(p-1)$. Hence, the Sylow $p$-subgroups of $GL_3(p)$ have order $p^3$. Set 
\begin{equation*} 
U := \left \lbrace \begin{pmatrix} 1 & a & b \\ 0 & 1 & c \\ 0 & 0 & 1 \end{pmatrix} \ : \ a, b, c \in \mathbb{F}_p \right\rbrace.
\end{equation*} 
Then $U$ is easily seen to be a subgroup of $GL_3(p)$ with order $p^3$. In other words, $U$ is a Sylow $p$-subgroup of $GL_3(p)$. Let $a,b,c \in \mathbb{F}_p$ and 
\begin{equation*}
	x := \begin{pmatrix} 1 & a & b \\ 0 & 1 & c \\ 0 & 0 & 1 \end{pmatrix} \in U.
\end{equation*}
One can show by induction that 
\begin{equation*}
x^k = \begin{pmatrix} 1 & ka & kb + \frac{(k-1)k}{2}ac \\ 0 & 1 & kc \\ 0 & 0 & 1 \end{pmatrix} 
\end{equation*}
for every positive integer $k$. In particular, $x^p$ is the identity matrix $I_3$. Since $x$ was an arbitrarily chosen element of $U$, it follows that $U$ has exponent $p$. Hence, any Sylow $p$-subgroup of $GL_3(p)$ has exponent $p$, and so (1) holds for $n = 3$. 

Now, let $n \in \lbrace 2,3 \rbrace$ and $H$ be a nontrivial cyclic $p$-subgroup of $GL_n(p)$. Then $H = \langle y \rangle$ for some nontrivial $p$-element $y$ of $GL_n(p)$. By (1), $y$ has order $p$. Thus $H = \langle y \rangle \cong C_p$, and so (2) holds. 
\end{proof}

Recall that a group $G$ is called \textit{minimal nonnilpotent} if any proper subgroup of $G$ is nilpotent, while $G$ itself is not nilpotent. 

\begin{lemma}
\label{lemma_minimal_non_nilpotent}
Let $G$ be a minimal nonnilpotent finite group and $S$ be a nonnormal Sylow subgroup of $G$. Then $S$ is cyclic. 
\end{lemma}

\begin{proof}
By \cite[Kapitel III, Satz 5.2]{Huppert}, we have $|G| = p^aq^b$ with distinct prime numbers $p$, $q$ and positive integers $a$, $b$, where $G$ has a normal Sylow $p$-subgroup and cyclic Sylow $q$-subgroups. Since $S$ is not normal in $G$, we have that $S$ is a Sylow $q$-subgroup of $G$. Consequently, $S$ is cyclic. 
\end{proof}

The next lemma is due to Oliver \cite{Oliver}. It will play an important role in the proof of Theorem \ref{third_main_result}. 

\begin{lemma}
\label{lemma_oliver}
(\cite[Lemma 1.11]{Oliver}) Let $p$ be a prime number, $A$ be finite abelian $p$-group and $G$ be a subgroup of $\mathrm{Aut}(A)$. Suppose that the following hold:
\begin{enumerate}
	\item[\textnormal{(1)}] The Sylow $p$-subgroups of $G$ have order $p$ and are not normal in $G$.
	\item[\textnormal{(2)}] For each $x \in G$ of order $p$, the group $[x,A]$ has order $p$. 
\end{enumerate}
Set $H := O^{p'}(G)$, $A_1 := C_A(H)$ and $A_2 := [H,A]$. Then $G$ normalizes $A_1$ and $A_2$, $A = A_1 \times A_2$, and $A_2 \cong C_p \times C_p$. 
\end{lemma}

Let $p$ be a prime number and $S$ be a finite $p$-group. Following \cite[§65]{BerkovichJanko}, we say that $S$ is an \textit{$\mathcal{A}_2$-group} if $S$ contains a nonabelian maximal subgroup, while any subgroup of $S$ with index $p^2$ is abelian. 

\begin{lemma}
\label{A2-groups} 
Let $p$ be a prime number and $S$ be a finite $p$-group. Suppose that $S$ is a nonmetacyclic $\mathcal{A}_2$-group and that $|S| > p^4$. Then the following hold: 
\begin{enumerate}
	\item[\textnormal{(1)}] (\cite[Proposition 71.4]{BerkovichJanko}) If $S$ possesses precisely one abelian maximal subgroup and if $S' \le Z(S)$, then $Z(S) = \Phi(S)$. 
	\item[\textnormal{(2)}] (\cite[Proposition 71.5]{BerkovichJanko}) If any maximal subgroup of $S$ is nonabelian and if $p$ is odd, then $|S| = p^5$. 
\end{enumerate}
\end{lemma}



\section{Proofs of the main results} 
\label{proofs} 
\begin{proof}[Proof of Theorem \ref{first_main_result}]
Suppose that the theorem is false, and let $G$ be a minimal counterexample. We will derive a contradiction in several steps. Set $\mathcal{F} := \mathcal{F}_S(G)$. 

\medskip
(1) \textit{Let $H < G$ with $S \cap H \in \mathrm{Syl}_p(H)$ and $|S \cap H| \ge p|D|$. Then $\mathcal{F}_{S \cap H}(H)$ is supersolvable.}

By hypothesis, any subgroup of $S \cap H$ with order $|D|$ or $p|D|$ is abelian and weakly pronormal in $G$. Applying \cite[Lemma 2.2 (2)]{Asaad}, we conclude that any subgroup of $S \cap H$ with order $|D|$ or $p|D|$ is abelian and weakly pronormal in $H$. The minimality of $G$ implies that $\mathcal{F}_{S \cap H}(H)$ is supersolvable. 

\medskip
(2) \textit{Let $Q \in \mathcal{E}_{\mathcal{F}}^{*}$. Then $|Q| \ge p|D|$, and if $Q$ is not normal in $G$, then $N_{\mathcal{F}}(Q)$ is supersolvable.}

Suppose that $|Q| < p|D|$. Then there is a subgroup $R$ of $S$ such that $|R| = p |D|$ and $Q < R$. By hypothesis, $R$ is abelian, and so we have $R \le C_S(Q)$. As a member of $\mathcal{E}_{\mathcal F}^{*}$, the subgroup $Q$ is $\mathcal{F}$-centric. It follows that $R \le C_S(Q) \le Q$, a contradiction. Thus $|Q| \ge p|D|$. 

Suppose now that $Q$ is not normal in $G$. Hence, $N_G(Q)$ is a proper subgroup of $G$. As $Q$ is fully $\mathcal{F}$-normalized, we have $S \cap N_G(Q) = N_S(Q) \in \mathrm{Syl}_p(N_G(Q))$. Also, $|N_S(Q)| \ge |Q| \ge p|D|$. Now, (1) implies that $N_{\mathcal{F}}(Q) = \mathcal{F}_{N_S(Q)}(N_G(Q))$ is supersolvable. 

\medskip
(3) $|O_p(G)| \ge p|D|$. 

Assume that there is no $Q \in \mathcal{E}_{\mathcal F}^{*}$ with $Q \trianglelefteq G$. Then, for each $Q \in \mathcal{E}_{\mathcal{F}}^{*}$, the fusion system $N_{\mathcal{F}}(Q)$ is supersolvable by (2), and Lemma \ref{criterion_FS_one} implies that $\mathcal{F}$ is supersolvable. This contradiction shows that there is some $Q \in \mathcal{E}_{\mathcal F}^{*}$ with $Q \trianglelefteq G$. Applying (2), we conclude that $|O_p(G)| \ge |Q| \ge p|D|$. 

\medskip 
(4) $O_p(G) \le Z_{\mathcal{U}}(G)$. 

Let $U$ be a subgroup of $O_p(G)$ with order $|D|$ or $p|D|$. By hypothesis, $U$ is weakly pronormal in $G$. Hence, there is a subgroup $K$ of $G$ such that $G = UK$ and such that $U \cap K$ is pronormal in $G$. By \cite[Theorem 4.3]{LiuYu}, $U \cap K$ is weakly $\mathcal{F}$-closed. For each $g \in G$, we have $(U \cap K)^g \le O_p(G) \le S$, and so it follows that $(U \cap K)^g = U \cap K$. Consequently, $U \cap K$ is normal in $G$, and therefore, $U$ is $c$-supplemented in $G$ in the sense of \cite{Asaad_c_supplemented, BBWG}. 

Since $U$ was arbitrarily chosen, any subgroup of $O_p(G)$ with order $|D|$ or $p|D|$ is $c$-supplemented in $G$. Applying \cite[Theorem 3.2 and Corollary 3.4]{Asaad_c_supplemented}, we conclude that $O_p(G) \le Z_{\mathcal U}(G)$. 

\medskip
(5) \textit{The final contradiction.}

If $H$ is a proper subgroup of $G$ with $O_p(G) < S \cap H$ and $S \cap H \in \mathrm{Syl}_p(H)$, then $\mathcal{F}_{S \cap H}(H)$ is supersolvable by (1) and (3). Also, $O_p(G) \le Z_{\mathcal U}(G)$ by (4). So $\mathcal{F}$ is supersolvable by Lemma \ref{criterion_FS_two}. This contradiction completes the proof. 
\end{proof}

\begin{proof}[Proof of Theorem \ref{second_main_result}]
Suppose that the theorem is false, and let $\mathcal{F}$ be a counterexample such that $|\mathcal F|$, the number of morphisms in $\mathcal F$, is minimal. We will derive a contradiction in several steps. In some parts of the proof, we argue similarly as in the proof of Theorem \ref{first_main_result}.

\medskip
(1) \textit{Let $R$ be a subgroup of $S$ with $|R| \ge 2|D|$ and $\mathcal{F}_0$ be a proper saturated subsystem of $\mathcal{F}$ on $R$. Then $\mathcal{F}_0$ is supersolvable.}

By hypothesis, any subgroup of $R$ with order $|D|$ is abelian and weakly $\mathcal{F}$-closed, and if $R$ is nonabelian, then we moreover have that any subgroup of $R$ with order $2|D|$ is abelian and weakly $\mathcal{F}$-closed. Clearly, any weakly $\mathcal{F}$-closed subgroup of $R$ is weakly $\mathcal{F}_0$-closed. Consequently, the fusion system $\mathcal{F}_0$ satisfies the hypotheses of the theorem, and the minimality of $\mathcal{F}$ implies that $\mathcal{F}_0$ is supersolvable. 

\medskip
(2) \textit{Let $Q \in \mathcal{E}_{\mathcal{F}}^{*}$. Then $|Q| \ge 2|D|$, and if $Q$ is not normal in $\mathcal{F}$, then $N_{\mathcal{F}}(Q)$ is supersolvable.} 

Suppose that $|Q| < 2|D|$. Then there is a subgroup $R$ of $S$ such that $|R| = 2|D|$ and $Q < R$. By hypothesis, $R$ is abelian, and so we have $R \le C_S(Q)$. As a member of $\mathcal{E}_{\mathcal{F}}^{*}$, the subgroup $Q$ is $\mathcal{F}$-centric. It follows that $R \le C_S(Q) \le Q$, a contradiction. Thus $|Q| \ge 2|D|$. 

Suppose now that $Q$ is not normal in $\mathcal{F}$. Then $N_{\mathcal{F}}(Q)$ is a proper saturated subsystem of $\mathcal{F}$ on $N_S(Q)$. Since $|N_S(Q)| \ge |Q| \ge 2|D|$, it follows from (1) that $N_{\mathcal{F}}(Q)$ is supersolvable. 

\medskip
(3) $|O_2(\mathcal{F})| \ge 2|D|$.

Assume that there is no $Q \in \mathcal{E}_{\mathcal{F}}^{*}$ with $Q \trianglelefteq \mathcal{F}$. Then, for each $Q \in \mathcal{E}_{\mathcal{F}}^{*}$, the fusion system $N_{\mathcal{F}}(Q)$ is supersolvable by (2), and Lemma \ref{criterion_FS_one} implies that $\mathcal{F}$ is supersolvable. This contradiction shows that there is some $Q \in \mathcal{E}_{\mathcal{F}}^{*}$ with $Q \trianglelefteq \mathcal{F}$. Applying (2), we conclude that $|O_2(\mathcal{F})| \ge |Q| \ge 2|D|$. 

\medskip
(4) \textit{There is a series $1 = S_0 \le S_1 \le \dots \le S_n = O_2(\mathcal{F})$ of subgroups of $O_2(\mathcal{F})$ such that $S_i$ is strongly $\mathcal{F}$-closed for all $0 \le i \le n$ and such that $S_{i+1}/S_i$ is cyclic for all $0 \le i < n$.}

Let $N := O_2(\mathcal{F})$, $A := \mathrm{Aut}_{\mathcal{F}}(N)$ and $G := N \rtimes A$ be the outer semidirect product of $N$ and $A$ with respect to the natural action of $A$ on $N$. Identifying each element of $N$ with its canonical image in $G$, we may regard $N$ as a normal subgroup of $G$. Likewise, identifying each element of $A$ with its canonical image in $G$, we may regard $A$ as a subgroup of $G$. Then $G = NA$, $N \cap A = 1$, and for all $x \in N$ and $\alpha \in A$, we have $\alpha^{-1}x\alpha = \alpha(x)$.

Let $P \le N$ such that $P$ is weakly $\mathcal{F}$-closed. We show that $P \trianglelefteq G$. Since $P$ is weakly $\mathcal{F}$-closed, we have $P \trianglelefteq S$ and thus $N \le N_G(P)$. For all $\alpha \in A$, we have $\alpha^{-1}P\alpha = \alpha(P) = P$, where the latter equality holds since $P$ is weakly $\mathcal{F}$-closed. Thus $A \le N_G(P)$, and it follows that $G = NA \le N_G(P)$. So we have $P \trianglelefteq G$, as wanted. 


By the preceding paragraph and by hypothesis, any subgroup of $N$ with order $|D|$ is normal in $G$, and if $N$ is nonabelian, then we moreover have that any subgroup of $N$ with order $2|D|$ is normal in $G$. Lemma \ref{lemma_skiba} implies that $N \le Z_{\mathcal{U}}(G)$. Hence, there is a series $1 = S_0 \le S_1 \le \dots \le S_n = N$ of subgroups of $N$ such that $S_i \trianglelefteq G$ for all $0 \le i \le n$ and such that $S_{i+1}/S_i$ is cyclic for all $0 \le i < n$. To complete the proof of (4), we show that $S_i$ is strongly $\mathcal{F}$-closed for all $0 \le i \le n$. Let $P$ be a subgroup of $S_i$ for some $0 \le i \le n$, and let $P_1$ be a subgroup of $S$ such that there is an $\mathcal{F}$-isomorphism $\varphi: P \rightarrow P_1$. Clearly, $N$ is strongly $\mathcal{F}$-closed, and therefore, we have $P_1 \le N$. Since $N$ is normal in $\mathcal{F}$, the isomorphism $\varphi$ extends to an automorphism $\alpha \in \mathrm{Aut}_{\mathcal{F}}(N) = A$. We have $P_1 = \varphi(P) = \alpha(P) \le \alpha(S_i) = \alpha^{-1}S_i\alpha = S_i$, where the last equality holds since $S_i$ is normal in $G$. We have shown that each $\mathcal{F}$-conjugate of $P$ lies in $S_i$. Therefore, $S_i$ is strongly $\mathcal{F}$-closed, as required. 

\medskip
(5) \textit{The final contradiction.} 

Applying Lemma \ref{criterion_FS_three}, we deduce from (1), (3) and (4) that $\mathcal{F}$ is supersolvable. This contradiction completes the proof.  
\end{proof} 	

\begin{remark}
Another proof of Theorem \ref{second_main_result} works as follows: One assumes that Theorem \ref{second_main_result} is false and considers a counterexample $\mathcal{F}$ such that $|\mathcal{F}|$ is minimal. Arguing as in the above proof of Theorem \ref{second_main_result}, one shows that there is some $Q \in \mathcal{E}_{\mathcal F}^{*}$ with $Q \trianglelefteq \mathcal{F}$. Then $\mathcal{F}$ is constrained, and the model theorem \cite[Part III, Theorem 5.10]{AKO} implies that $\mathcal{F}$ is the $2$-fusion system of a finite group $G$. Theorem \ref{theorem_asaad_two} shows that $G$ is $2$-nilpotent, and so it follows that $\mathcal{F}$ is nilpotent, a contradiction. This proof of Theorem \ref{second_main_result} has the advantage of being shorter than the above proof of Theorem \ref{second_main_result}, but it has the disadvantage of being less elementary since it depends on the model theorem, which is a quite deep result. 
\end{remark}

\begin{proof}[Proof of Theorem \ref{third_main_result}]
Suppose that the theorem is false, and let $\mathcal{F}$ be a counterexample such that $|\mathcal{F}|$ is minimal. We will derive a contradiction in several steps. 

\medskip
(1) \textit{$S$ is not normal in $\mathcal{F}$.}

Assume that $S \trianglelefteq \mathcal{F}$. We claim that there is a $p$-closed finite group $G$ with $S \in \mathrm{Syl}_p(G)$ and $\mathcal{F} = \mathcal{F}_S(G)$. Indeed, the normality of $S$ in $\mathcal{F}$ implies that $\mathcal{F}$ is constrained, and so the existence of such a group $G$ follows from the model theorem \cite[Part III, Theorem 5.10]{AKO}. Alternatively, without using the model theorem, one can argue as follows: Let $A := \mathrm{Aut}_{\mathcal{F}}(S)$. Since $\mathcal{F}$ is saturated, $\mathrm{Inn}(S)$ is a Sylow $p$-subgroup of $A$, whence $|\mathrm{Inn}(S)|$ and $|A : \mathrm{Inn}(S)|$ are coprime. Since $\mathrm{Inn}(S) \trianglelefteq A$, the Schur-Zassenhaus theorem \cite[Theorem 3.8]{Isaacs} implies that $\mathrm{Inn}(S)$ has a complement $H$ in $A$. Let $G = S \rtimes H$ be the outer semidirect product of $S$ and $H$ with respect to the natural action of $H$ on $S$. Identifying each element of $S$ with its canonical image in $G$, we may regard $S$ as a normal subgroup of $G$. Then we have $S \in \mathrm{Syl}_p(G)$ since $|H|$ is not divisible by $p$, and it is not hard to show that $\mathcal{F} = \mathcal{F}_S(G)$. 

Since $S$ is normal in $G$ and since any maximal subgroup of $S$ is weakly $\mathcal{F}$-closed, we have that any maximal subgroup of $S$ is normal in $G$. Lemma \ref{lemma_skiba} implies that $S \le Z_{\mathcal{U}}(G)$. Hence, there is a series $1 = S_0 \le S_1 \le \dots \le S_m = S$ of subgroups of $S$ such that $S_i \trianglelefteq G$ for all $0 \le i \le m$ and such that $S_{i+1}/S_i$ is cyclic for all $0 \le i < m$. In particular, $S_i$ is strongly $\mathcal{F}$-closed for all $0 \le i \le m$, and so it follows that $\mathcal{F}$ is supersolvable. This contradiction shows that $S$ is not normal in $\mathcal{F}$. 

\medskip
(2) \textit{$S$ possesses more than one abelian maximal subgroup, and any abelian maximal subgroup of $S$ contains $Z(S)$. Moreover, we have $|S'| = p$ and $|S : Z(S)| = p^2$.}

Assume that $S$ is abelian. Then $S \trianglelefteq \mathcal{F}$ by \cite[Part I, Corollary 4.7 (a)]{AKO}, which contradicts (1). So $S$ must be nonabelian. In particular, $S$ is not cyclic, and so $S$ has more than one abelian maximal subgroup by hypothesis. If $R$ is an abelian maximal subgroup of $S$, then $Z(S) \le R$ because otherwise $S = RZ(S)$ would be abelian. We see from \cite[Lemma 1.9]{Oliver} that $|S'| = p$ and $|S : Z(S)| = p^2$.

\medskip
(3) \textit{$O_p(\mathcal{F})$ is an abelian maximal subgroup of $S$, we have $\mathcal{E}_{\mathcal{F}}^{*} = \lbrace O_p(\mathcal{F}), S \rbrace$, and the subgroups $O_p(\mathcal{F})$ and $S$ are precisely the $\mathcal{F}$-centric, $\mathcal{F}$-radical subgroups of $S$.} 

First, we argue that any member of $\mathcal{E}_{\mathcal{F}}^{*} \setminus \lbrace S \rbrace$ is an abelian maximal subgroup of $S$. Let $Q \in \mathcal{E}_{\mathcal{F}}^{*} \setminus \lbrace S \rbrace$. Then, since $Q$ is $\mathcal{F}$-centric, we have $Z(S) < Q < S$. As $|S : Z(S)| = p^2$ by (2), it follows that $Q$ is a maximal subgroup of $S$. Also, $Q$ is abelian since $|Q : Z(S)| = p$. 

Let $Q \in \mathcal{E}_{\mathcal{F}}^{*}$ such that $Q$ is not normal in $\mathcal{F}$. Then $N_{\mathcal{F}}(Q)$ is a proper saturated subsystem of $\mathcal{F}$ on $S$, and it is easy to see that $N_{\mathcal{F}}(Q)$ satisfies the hypotheses of the theorem. So $N_{\mathcal{F}}(Q)$ is supersolvable by the minimality of $\mathcal{F}$. If no member of $\mathcal{E}_{\mathcal{F}}^{*}$ is normal in $\mathcal{F}$, then $N_{\mathcal{F}}(Q)$ is supersolvable for each $Q \in \mathcal{E}_{\mathcal{F}}^{*}$, whence $\mathcal{F}$ is supersolvable by Lemma \ref{criterion_FS_one}. Since this is not the case, there must exist some $Q \in \mathcal{E}_{\mathcal{F}}^{*}$ with $Q \trianglelefteq \mathcal{F}$. We have $Q \ne S$ by (1), and so the preceding paragraph implies that $Q$ is an abelian maximal subgroup of $S$. Since $Q \le O_p(\mathcal{F})$ and $O_p(\mathcal{F}) \ne S$, we have $O_p(\mathcal{F}) = Q$. Hence, we have shown that $O_p(\mathcal{F})$ is an abelian maximal subgroup of $S$. 

Clearly, $S$ is $\mathcal{F}$-centric and $\mathcal{F}$-radical. Since $O_p(\mathcal{F})$ is $\mathcal{F}$-essential, we also have that $O_p(\mathcal{F})$ is $\mathcal{F}$-centric and $\mathcal{F}$-radical (see \cite[Part I, Proposition 3.3 (a)]{AKO}). Conversely, if $R$ is an $\mathcal{F}$-centric, $\mathcal{F}$-radical subgroup of $S$, then $O_p(\mathcal{F}) \le R$ by \cite[Part I, Proposition 4.5]{AKO}, and the maximality of $O_p(\mathcal{F})$ in $S$ implies that $R = O_p(\mathcal{F})$ or $R = S$. Consequently, the subgroups $O_p(\mathcal{F})$ and $S$ are precisely the $\mathcal{F}$-centric, $\mathcal{F}$-radical subgroups of $S$. 

Since any member of $\mathcal{E}_{\mathcal{F}}^{*}$ is $\mathcal{F}$-centric and $\mathcal{F}$-radical, it follows that $\mathcal{E}_{\mathcal{F}}^{*} \subseteq \lbrace O_p(\mathcal{F}), S \rbrace$. The other inclusion also holds, and so we have $\mathcal{E}_{\mathcal{F}}^{*} = \lbrace O_p(\mathcal{F}), S \rbrace$. 

\medskip(4) \textit{There is no nontrivial subgroup of $Z(S)$ which is normal in $\mathcal{F}$.}

Assume that there is a subgroup $1 \ne Z \le Z(S)$ with $Z \trianglelefteq \mathcal{F}$. From Lemma \ref{lemma_strongly_closed_1}, we see that any maximal subgroup of $S/Z$ is weakly $\mathcal{F}/Z$-closed. By (2), $S$ possesses more than one abelian maximal subgroup, and each of them contains $Z$. Hence, $S/Z$ has more than one abelian maximal subgroup. Consequently, the fusion system $\mathcal{F}/Z$ satisfies the hypotheses of the theorem, and so $\mathcal{F}/Z$ is supersolvable by the minimality of $\mathcal{F}$. 

Let $R_1$ and $R_2$ be two distinct abelian maximal subgroups of $S$. Since $R_1/Z$ and $R_2/Z$ are weakly $\mathcal{F}/Z$-closed and since $\mathcal{F}/Z$ is supersolvable, we see from Lemma \ref{lemma_strongly_closed_3} that $R_1/Z$ and $R_2/Z$ are in fact strongly $\mathcal{F}/Z$-closed. Lemma \ref{lemma_strongly_closed_2} implies that $R_1$ and $R_2$ are strongly $\mathcal{F}$-closed. Since $R_1$ and $R_2$ are abelian, it follows from \cite[Part I, Corollary 4.7 (a)]{AKO} that $R_1$ and $R_2$ are normal in $\mathcal{F}$. Hence, $S = R_1R_2$ is normal in $\mathcal{F}$, which contradicts (1). Therefore, there is no subgroup $1 \ne Z \le Z(S)$ with $Z \trianglelefteq \mathcal{F}$. 

\medskip
(5) \textit{The Sylow $p$-subgroups of $\mathrm{Aut}_{\mathcal{F}}(O_p(\mathcal{F}))$ have order $p$ and are not normal in $\mathrm{Aut}_{\mathcal{F}}(O_p(\mathcal{F}))$.}

As $\mathcal{F}$ is saturated and $O_p(\mathcal{F})$ is fully $\mathcal{F}$-normalized, $\mathrm{Aut}_S(O_p(\mathcal{F}))$ is a Sylow $p$-subgroup of $\mathrm{Aut}_{\mathcal{F}}(O_p(\mathcal{F}))$. Since $O_p(\mathcal{F})$ is $\mathcal{F}$-centric, abelian and maximal in $S$ by (3), we have $\mathrm{Aut}_S(O_p(\mathcal{F})) \cong S/O_p(\mathcal{F}) \cong C_p$. Since $O_p(\mathcal{F})$ is $\mathcal{F}$-radical by (3), we also have $1 = \mathrm{Inn}(O_p(\mathcal{F})) = O_p(\mathrm{Aut}_{\mathcal{F}}(O_p(\mathcal{F})))$. Hence, $\mathrm{Aut}_S(O_p(\mathcal{F}))$ is not normal in $\mathrm{Aut}_{\mathcal{F}}(O_p(\mathcal{F}))$.

\medskip
(6) \textit{Assume that $\alpha \in \mathrm{Aut}_{\mathcal{F}}(O_p(\mathcal{F}))$ has order $p$. Then $[\alpha,O_p(\mathcal{F})]$ has order $p$.}

Since $\mathrm{Aut}_S(O_p(\mathcal{F}))$ is a Sylow $p$-subgroup of $\mathrm{Aut}_{\mathcal{F}}(O_p(\mathcal{F}))$, there is some $\beta \in \mathrm{Aut}_{\mathcal{F}}(O_p(\mathcal{F}))$ such that $\gamma := \beta^{-1}\alpha\beta \in \mathrm{Aut}_S(O_p(\mathcal{F}))$. A direct calculation shows that $[\gamma,O_p(\mathcal{F})] = \beta([\alpha,O_p(\mathcal{F})])$, whence $|[\gamma,O_p(\mathcal{F})]| = |[\alpha,O_p(\mathcal{F})]|$. Therefore, we may assume without loss of generality that $\alpha \in \mathrm{Aut}_S(O_p(\mathcal{F}))$. Then there is some $s \in S$ with $\alpha(x) = x^s$ for all $x \in O_p(\mathcal{F})$, and we have $[\alpha,O_p(\mathcal{F})] = [s,O_p(\mathcal{F})] \le S'$. Since $S'$ has order $p$ by (2) and since $\alpha$ acts nontrivially on $O_p(\mathcal{F})$, it follows that $[\alpha,O_p(\mathcal{F})]$ has order $p$, as wanted. 

\medskip
(7) \textit{$S$ is extraspecial of order $p^3$ and exponent $p$.}

Set $A := O_p(\mathcal{F})$, $G := \mathrm{Aut}_{\mathcal{F}}(A)$, $H := O^{p'}(G)$, $A_1 := C_A(H)$ and $A_2 := [H,A]$. By (3), (5), (6) and Lemma \ref{lemma_oliver}, $G$ normalizes $A_1$ and $A_2$, $A = A_1 \times A_2$ and $A_2 \cong C_p \times C_p$.  

Since $\mathrm{Aut}_S(A)$ is a Sylow $p$-subgroup of $G$, we have $\mathrm{Aut}_S(A) \le H$. Consequently, $\mathrm{Aut}_S(A)$ centralizes $A_1$. In other words, we have $A_1 \le Z(S)$.

We show now that $A_1 \trianglelefteq \mathcal{F}$. Since $G$ normalizes $A_1$ and since $\mathcal{E}_{\mathcal{F}}^{*} = \lbrace A, S \rbrace$ by (3), it suffices to show that $A_1$ is $\mathrm{Aut}_{\mathcal{F}}(S)$-invariant (see \cite[Part I, Proposition 4.5]{AKO}). But this follows from the fact that $G$ normalizes $A_1$ because any $\mathcal{F}$-automorphism of $S$ restricts to an $\mathcal{F}$-automorphism of $A$.

We have shown that $A_1$ is a subgroup of $Z(S)$ which is normal in $\mathcal{F}$. Therefore, by (4), $A_1$ must be trivial. It follows that $A = A_2 \cong C_p \times C_p$. Since $A$ is maximal in $S$ by (3), it further follows that $|S| = p^3$. By (2), $S$ is nonabelian, and since any nonabelian $p$-group of order $p^3$ is extraspecial, it follows that $S$ is extraspecial. 

It remains to show that $S$ has exponent $p$. Assume that this is not the case. Then $S$ has exponent $p^2$, and hence, $S$ has a cyclic maximal subgroup. So $S$ is metacyclic, and \cite[Theorem C]{BB} implies that $\mathcal{F}$ is supersolvable. This contradiction shows that $S$ has exponent $p$.   

\medskip
(8) \textit{The final contradiction.} 

By (7), $S$ is extraspecial of order $p^3$ and exponent $p$. As a consequence of (3), $S$ has precisely one elementary abelian $\mathcal{F}$-centric and $\mathcal{F}$-radical subgroup. Applying \cite[Theorem 1.1]{RuizViruel}, or only \cite[Lemma 4.7]{RuizViruel}, we conclude that $\mathcal{F}$ is isomorphic to one of the fusion systems considered in Lemma \ref{lemma_extraspecial}. Alternatively, this can be seen from \cite[Lemma 9.2]{Craven}. Therefore, by Lemma \ref{lemma_extraspecial}, there is a maximal subgroup of $S$ which is not weakly $\mathcal{F}$-closed. On the other hand, we have by hypothesis that every maximal subgroup of $S$ is weakly $\mathcal{F}$-closed. With this contradiction, the proof is complete. 
\end{proof}

For the proof of Theorem \ref{fourth_main_result}, we need the following lemma. It verifies Theorem \ref{fourth_main_result} for the case that the fusion system $\mathcal{F}$ in the statement of the theorem is realized by a finite group. 

\begin{lemma}
\label{lemma_Thm_D} 
Let $G$ be a finite group, $p$ be an odd prime number, $S$ be a Sylow $p$-subgroup of $G$ and $\mathcal{F} := \mathcal{F}_S(G)$. Suppose that there is a subgroup $D$ of $S$ with $1 < D < S$ such that any subgroup of $S$ with order $|D|$ is abelian and weakly $\mathcal{F}$-closed. Then $\mathcal{F}$ is supersolvable. 
\end{lemma}

\begin{proof}
Suppose that the lemma is false, and let $G$ be a minimal counterexample. We will derive a contradiction in several steps. 

\medskip 
(1) \textit{Let $H < G$ with $S \cap H \in \mathrm{Syl}_p(H)$ and $|S \cap H| > |D|$. Then $\mathcal{F}_{S \cap H}(H)$ is supersolvable.}

By hypothesis, any subgroup of $S \cap H$ with order $|D|$ is abelian and weakly $\mathcal{F}$-closed. Clearly, any weakly $\mathcal{F}$-closed subgroup of $S \cap H$ is weakly $\mathcal{F}_{S \cap H}(H)$-closed. Hence, any subgroup of $S \cap H$ with order $|D|$ is abelian and weakly $\mathcal{F}_{S \cap H}(H)$-closed. The minimality of $G$ implies that $\mathcal{F}_{S \cap H}(H)$ is supersolvable. 

\medskip 
(2) $C_G(O_p(G)) \le O_p(G)$.

Let $Q \in \mathcal{E}_{\mathcal{F}}^{*}$. Since $C_S(Q) \le Q$ and since any subgroup of $S$ with order $|D|$ is abelian, we have $|Q| \ge |D|$. This implies that $|N_S(Q)| > |D|$. Indeed, if $Q \ne S$, then $|N_S(Q)| > |Q| \ge |D|$, and if $Q = S$, then $|N_S(Q)| = |S| > |D|$. 

Assume that $Q$ is not normal in $G$. Hence $N_G(Q) < G$. We have $S \cap N_G(Q) = N_S(Q) \in \mathrm{Syl}_p(N_G(Q))$ since $Q$ is fully $\mathcal{F}$-normalized and $|S \cap N_G(Q)| = |N_S(Q)| > |D|$ by the preceding paragraph. So, by (1), the fusion system $N_{\mathcal{F}}(Q) = \mathcal{F}_{N_S(Q)}(N_G(Q))$ is supersolvable. 

If no member of $\mathcal{E}_{\mathcal{F}}^{*}$ is normal in $G$, then $N_{\mathcal{F}}(Q)$ is supersolvable for each $Q \in \mathcal{E}_{\mathcal{F}}^{*}$, whence $\mathcal{F}$ is supersolvable by Lemma \ref{criterion_FS_one}. Since this is not the case, there must exist some $Q \in \mathcal{E}_{\mathcal{F}}^{*}$ with $Q \trianglelefteq G$. Hence, $\mathcal{F}$ is constrained, and \cite[Proposition 8.8]{Hethelyi} shows that the model of $\mathcal{F}$ is isomorphic to $G/O_{p'}(G)$.  

If $O_{p'}(G) \ne 1$, then $\mathcal{F}_{SO_{p'}(G)/O_{p'}(G)}(G/O_{p'}(G))$ is supersolvable by the minimality of $G$, whence $\mathcal{F}$ is supersolvable. Thus $O_{p'}(G) = 1$. Consequently, $G$ is the model of $\mathcal{F}$, and so we have $C_G(O_p(G)) \le O_p(G)$, as wanted.  

\medskip
(3) $|O_p(G)| = |D|$. 

Assume that $|O_p(G)| > |D|$. By hypothesis, any subgroup of $O_p(G)$ with order $|D|$ is weakly $\mathcal{F}$-closed. Clearly, any weakly $\mathcal{F}$-closed subgroup of $O_p(G)$ is normal in $G$. Therefore, any subgroup of $O_p(G)$ with order $|D|$ is normal in $G$. So we have $O_p(G) \le Z_{\mathcal U}(G)$ by Lemma \ref{lemma_skiba}. If $H$ is a proper subgroup of $G$ with $O_p(G) < S \cap H$ and $S \cap H \in \mathrm{Syl}_p(H)$, then $\mathcal{F}_{S \cap H}(H)$ is supersolvable by (1). So, by Lemma \ref{criterion_FS_two}, $\mathcal{F}$ is supersolvable. This contradiction shows that $|O_p(G)| \le |D|$. 

Because of (2), we also have $|O_p(G)| \ge |D|$. Thus $|O_p(G)| = |D|$, as desired.  

\medskip
(4) \textit{$O_p(G)$ is elementary abelian.}

Assume that $O_p(G)$ is not elementary abelian. Then $\Phi(O_p(G)) \ne 1$ by \cite[Lemma 4.5]{Isaacs}. Set $\overline{G} := G/\Phi(O_p(G))$ and $\overline{\mathcal{F}} := \mathcal{F}/\Phi(O_p(G)) = \mathcal{F}_{\overline S}(\overline G)$. By hypothesis and by Lemma \ref{lemma_strongly_closed_1}, any subgroup of $\overline{S}$ with order $\frac{|D|}{|\Phi(O_p(G))|}$ is abelian and weakly $\overline{\mathcal{F}}$-closed. The minimality of $G$ implies that $\overline{\mathcal{F}}$ is supersolvable. 

By Lemma \ref{lemma_supersolvable_FS_2} (1), there is a series $\Phi(O_p(G)) = V_0 \le V_1 \le \dots \le V_m = S$ of subgroups of $S$ such that $\overline{V_i}$ is strongly $\overline{\mathcal{F}}$-closed for all $0 \le i \le m$, such that the quotient $\overline{V_{i+1}}/\overline{V_i}$ is cyclic for all $0 \le i < m$ and such that $V_j = O_p(G)$ for some $0 < j < m$. Since $\overline{V_i} \trianglelefteq \overline{G}$ for all $0 \le i \le j$, it follows that $\overline{O_p(G)} \le Z_{\mathcal{U}}(\overline{G})$. Applying \cite[Chapter 1, Theorem 7.19]{Weinstein}, we conclude that $O_p(G) \le Z_{\mathcal{U}}(G)$.



By (1) and (3), for any proper subgroup $H$ of $G$ with $O_p(G) < S \cap H$ and $S \cap H \in \mathrm{Syl}_p(H)$, the fusion system $\mathcal{F}_{S \cap H}(H)$ is supersolvable. Lemma \ref{criterion_FS_two} implies that $\mathcal{F}$ is supersolvable. This contradiction shows that $O_p(G)$ is elementary abelian. 

\medskip
(5) \textit{If $O_p(G) \le H < G$, then $H$ is $p$-closed.}

Let $O_p(G) \le H < G$. Without loss of generality, we assume that $S \cap H \in \mathrm{Syl}_p(H)$. Clearly, $O_p(G) \le S \cap H$, and if $O_p(G) = S \cap H$, then $H$ is $p$-closed. Assume now that $O_p(G) < S \cap H$. Then $\mathcal{E} := \mathcal{F}_{S \cap H}(H)$ is supersolvable by (1) and (3). So, by \cite[Proposition 2.3]{Su}, $S \cap H$ is normal in $\mathcal{E}$. Thus $\mathcal{E} = N_{\mathcal{E}}(S \cap H) = \mathcal{F}_{S \cap H}(N_H(S \cap H))$. For each $h \in H$, let $c_h$ denote the automorphism of $O_p(G)$ induced by conjugation with $h$, i.e. 
\begin{equation*}
c_h: O_p(G) \rightarrow O_p(G), x \mapsto x^h.
\end{equation*}
Let $h \in H$. Then $c_h$ is a morphism in $\mathcal{E} = \mathcal{F}_{S \cap H}(N_H(S \cap H))$, and so we have $c_h = c_u$ for some $u \in N_H(S \cap H)$. Then $hu^{-1} \le C_G(O_p(G)) \le O_p(G) \le N_H(S \cap H)$ by (2). It follows that $h \in N_H(S \cap H)$, and since $h$ was an arbitrarily chosen element of $H$, we can conclude that $S \cap H \trianglelefteq H$. Hence, $H$ is $p$-closed, as desired. 

\medskip
(6) \textit{$S/O_p(G)$ is cyclic.}

Set $\overline{G} := G/O_p(G)$. Let $O_p(G) \le H < G$. Then $H$ is $p$-closed by (5). Hence, $\overline{H}$ is $p$-closed. The group $\overline{G}$ is not $p$-closed because otherwise $S$ would be normal in $G$, which is not true because of (3). We have shown that $\overline{G}$ is a non-$p$-closed group all of whose proper subgroups are $p$-closed. In other words, $\overline{G}$ is minimal non-$p$-closed in the sense of \cite{LG}. Applying \cite[Lemma 1]{LG}, we conclude that $\overline{G}$ is minimal nonnilpotent or that $\overline{G}/\Phi(\overline{G})$ is nonabelian simple. 

If $\overline{G}$ is minimal nonnilpotent, then $\overline{S}$ is cyclic by Lemma \ref{lemma_minimal_non_nilpotent}. 

We assume now that $\widehat{G} := \overline{G}/\Phi(\overline{G})$ is nonabelian simple, and we show that $\overline{S}$ is cyclic also in this case. Since $\Phi(\overline{G})$ is nilpotent and $O_p(\overline{G}) = 1$, we have that $\Phi(\overline{G})$ is a $p'$-group. Thus, $\overline{S}$ is isomorphic to the Sylow $p$-subgroups of $\widehat{G}$, and therefore, it suffices to show that $\widehat{G}$ has cyclic Sylow $p$-subgroups. 

Let $O_p(G) \le L < G$ such that $\overline{L} = \Phi(\overline{G})$. If $L \le H < G$, then $H$ is $p$-closed by (5), whence $H/L$ is $p$-closed. Consequently, any proper subgroup of $G/L \cong \widehat{G}$ is $p$-closed. Since $\widehat{G}$ is simple, the only proper quotient of $\widehat{G}$ is $\widehat{G}/\widehat{G}$, which is of course $p$-closed. The simplicity of $\widehat{G}$ also implies that $\widehat{G}$ is not $p$-closed. We have shown that $\widehat{G}$ is a non-$p$-closed group all of whose proper subgroups and proper quotients are $p$-closed. Hence, $\widehat{G}$ is minimal non-$p$-closed in the sense of \cite{Kappe} (note that the definition of a minimal non-$p$-closed group used in \cite{Kappe} does not coincide with the one used in \cite{LG}). Applying \cite[Theorem 3.5]{Kappe}, we conclude that $\widehat{G}$ has cyclic Sylow $p$-subgroups, as required. 

\medskip
(7) \textit{$O_p(G)$ is not maximal in $S$.}

Assume that $O_p(G)$ is maximal in $S$. Then, by (3) and by hypothesis, any maximal subgroup of $S$ is abelian and weakly $\mathcal{F}$-closed. Theorem \ref{third_main_result} implies that $\mathcal{F}$ is supersolvable, which is a contradiction. Hence, $O_p(G)$ is not maximal in $S$. 

\medskip
(8) $|O_p(G)| \ge p^4$. 

By (2) and (4), we have $C_S(O_p(G)) = O_p(G)$. Hence, $S/O_p(G)$ is isomorphic to a $p$-subgroup of $\mathrm{Aut}(O_p(G))$. By (6), $S/O_p(G)$ is cyclic. We have $|S/O_p(G)| \ge p^2$ since $O_p(G) \ne S$ by (3) and since $O_p(G)$ is not maximal in $S$ by (7). 

Now, let $e \in \mathbb{N}$ such that $|O_p(G)| = p^e$. We have to show that $e \ge 4$. As $|O_p(G)| = |D| > 1$ by (3), we have $e \ne 0$. Since $O_p(G)$ is elementary abelian by (4), we have $\mathrm{Aut}(O_p(G)) \cong GL_e(p)$. If $e = 1$, then $\mathrm{Aut}(O_p(G)) \cong GL_1(p) \cong C_{p-1}$ does not have any nontrivial $p$-subgroups, and if $e \in \lbrace 2,3 \rbrace$, then we see from Lemma \ref{Sylow_GL} (2) that $\mathrm{Aut}(O_p(G)) \cong GL_e(p)$ does not have any cyclic $p$-subgroups of order greater than $p$. Since $S/O_p(G)$ is isomorphic to a cyclic $p$-subgroup of $\mathrm{Aut}(O_p(G))$ of order at least $p^2$, it follows that $e \ge 4$, as required.  

\medskip
(9) \textit{The final contradiction.} 

Since $|S/O_p(G)| \ge p^2$ by (3) and (7), there exists $O_p(G) < T \le S$ with $|T/O_p(G)| = p^2$. Then $O_p(G)$ is properly contained in a maximal subgroup $T_1$ of $T$, and since $C_{T_1}(O_p(G)) \le C_G(O_p(G)) \le O_p(G)$ by (2), we have that $T_1$ is nonabelian. By (3) and by hypothesis, any subgroup of $T$ with index $p^2$ is abelian. Hence, $T$ is an $\mathcal{A}_2$-group in the sense of the definition given before Lemma \ref{A2-groups}. We have $|T| > |O_p(G)| \ge p^4$ by (8). 

Assume that $T$ is metacyclic. Then $O_p(G)$ is metacyclic as well. At the same time, $O_p(G)$ is elementary abelian by (4). It is rather easy to see that an elementary abelian finite $p$-group can only be metacyclic when its order is at most $p^2$. So it follows that $|O_p(G)| \le p^2$. This contradicts (8), and therefore, $T$ is not metacyclic. 

Assume that $T$ has no abelian maximal subgroups. Then $|T| = p^5$ by Lemma \ref{A2-groups} (2) and hence $|O_p(G)| = p^3$, which is a contradiction to (8). Thus, there exists an abelian maximal subgroup of $T$. 

Assume that there is only one abelian maximal subgroup of $T$, say $U$. We claim that $T' \le Z(T)$. Because of (2), we have $O_p(G) \not\le U$ and hence $T = O_p(G)U$. Set $Z := O_p(G) \cap U$. Since $O_p(G)$ and $U$ are abelian, the subgroup $Z$ is centralized by both $O_p(G)$ and $U$. This implies that $Z \le Z(T)$. We have $O_p(G)/Z \cong O_p(G)U/U = T/U \cong C_p$, and so $Z$ is maximal in $O_p(G)$. As a consequence of (2), $Z(T)$ is a proper subgroup of $O_p(G)$. So we have $Z \le Z(T) < O_p(G)$, and the maximality of $Z$ in $O_p(G)$ implies that $Z = Z(T)$. Therefore, $O_p(G)/Z(T)$ is a normal subgroup of $T/Z(T)$ with order $p$ and hence a central subgroup of $T/Z(T)$. The corresponding factor group $(T/Z(T))/(O_p(G)/Z(T)) \cong T/O_p(G)$ is cyclic by (6). It follows that $T/Z(T)$ is abelian. Thus $T' \le Z(T)$, as claimed above. Applying Lemma \ref{A2-groups} (1), we conclude that $\Phi(T) = Z(T) < O_p(G)$. So $T/O_p(G)$ is elementary abelian by \cite[Lemma 4.5]{Isaacs}, and since $|T/O_p(G)| = p^2$, it follows that $T/O_p(G)$ is not cyclic. This is a contradiction to (6). Therefore, $T$ has more than one abelian maximal subgroup. 

Now, applying \cite[Lemma 1.9]{Oliver}, we conclude that $Z(T)$ has index $p^2$ in $T$. Since $O_p(G)$ also has index $p^2$ in $T$ and $Z(T) \le C_G(O_p(G)) \le O_p(G)$ by (2), it follows that $O_p(G) = Z(T)$. In particular, we have $C_G(O_p(G)) \not\le O_p(G)$. This contradicts (2), and the proof is complete with this contradiction. 
\end{proof}

With Lemma \ref{lemma_Thm_D} at hand, we are able to prove Theorem \ref{fourth_main_result} in few lines. Our proof strongly relies on the model theorem \cite[Part III, Theorem 5.10]{AKO}.

\begin{proof}[Proof of Theorem \ref{fourth_main_result}]
Suppose that the theorem is false, and let $\mathcal{F}$ be a counterexample such that $|\mathcal{F}|$ is minimal. 

Let $Q$ be a member of $\mathcal{E}_{\mathcal{F}}^{*}$ which is not normal in $\mathcal{F}$. Since $C_S(Q) \le Q$ and since any subgroup of $S$ with order $|D|$ is abelian, we have $|Q| \ge |D|$. This implies that $|N_S(Q)| > |D|$. Hence, $N_{\mathcal{F}}(Q)$ is a proper saturated subsystem of $\mathcal{F}$ on $N_S(Q)$ satisfying the hypotheses of the theorem. Thus, $N_{\mathcal{F}}(Q)$ is supersolvable by the minimality of $\mathcal{F}$. 

If no member of $\mathcal{E}_{\mathcal{F}}^{*}$ is normal in $\mathcal{F}$, then $\mathcal{F}$ is supersolvable by the preceding paragraph and Lemma \ref{criterion_FS_one}. Therefore, there must be a member of $\mathcal{E}_{\mathcal{F}}^{*}$ which is normal in $\mathcal{F}$. Hence, $\mathcal{F}$ is constrained. The model theorem \cite[Part III, Theorem 5.10]{AKO} implies that there is a finite group $G$ with $S \in \mathrm{Syl}_p(G)$ and $\mathcal{F} = \mathcal{F}_S(G)$. Then Lemma \ref{lemma_Thm_D} implies that $\mathcal{F}$ is supersolvable. This contradiction completes the proof. 
\end{proof}

\noindent\textbf{Declarations of interest:} none.

\end{document}